\documentclass[a4paper,12pt]{amsart}

\usepackage{amsmath}
\usepackage{amssymb}
\usepackage{mathrsfs}
\usepackage{enumerate}
\usepackage{ifthen}
\usepackage{graphicx}
\usepackage{caption}
\usepackage[T1]{fontenc} 
\usepackage{tikz}

\nonstopmode \numberwithin{equation}{section}
\newtheorem{theorem}{Theorem}[section]
\newtheorem{corollary}{Corollary}[section]
\newtheorem{proposition}{Proposition}[section]
\newtheorem{lemma}{Lemma}[section]
\newtheorem{problem}{Problem}[section]

\theoremstyle{remark}
\theoremstyle{definition}

\theoremstyle{plain}
\numberwithin{equation}{section}
\numberwithin{theorem}{section}

\newtheorem{innercustomthm}{Theorem}
\newenvironment{customthm}[1]
  {\renewcommand\theinnercustomthm{#1}\innercustomthm}
  {\endinnercustomthm}

\begin{document}

\title{An application of the Schur algorithm to variability regions of certain analytic functions}

\author{Md Firoz Ali}
\address{Md Firoz Ali,
Department of Mathematics,
National Institute of Technology Durgapur,
 Durgapur - 713209,
West Bengal,
India.}
\email{ali.firoz89@gmail.com}

\author{Vasudevarao Allu}
\address{Vasudevarao Allu,
School of Basic Science,
Indian Institute of Technology Bhubaneswar,
Bhubaneswar-752 050, Odisha, India.}
\email{avrao@iitbbs.ac.in}

\author{Hiroshi Yanagihara}
\address{Hiroshi Yanagihara,
    Department of Applied Science,
	Faculty of Engineering,
	Yamaguchi University,
	Tokiwadai, Ube 755,
	Japan}
\email{hiroshi@yamaguchi-u.ac.jp}



\begin{abstract}
Let $\Omega $ be a convex domain in the complex plane ${\mathbb C}$ with $\Omega  \not= {\mathbb C}$, and $P$ be a conformal map of the unit disk ${\mathbb D}$ onto $\Omega$. Let ${\mathcal F}_\Omega$ be the class of analytic functions $g$ in ${\mathbb D}$ with $g({\mathbb D}) \subset \Omega$, and $H_1^\infty ({\mathbb D})$ be the closed unit ball of the Banach space $H^\infty ({\mathbb D})$ of bounded analytic functions $\omega$ in ${\mathbb D}$, with  norm $\| \omega \|_\infty = \sup_{z \in {\mathbb D}} |\omega (z)|$. Let ${\mathcal C}(n) = \{ (c_0,c_1 , \ldots , c_n ) \in {\mathbb C}^{n+1}: \text{there exists} \; \omega \in H_1^\infty ({\mathbb D}) \;  \text{satisfying} \; \omega (z) = c_0+c_1z + \cdots + c_n z^n + \cdots$ for ${z\in \mathbb D}\}$. For each fixed $z_0 \in {\mathbb D}$, $j=-1,0,1,2, \ldots$ and $c  = (c_0, c_1 , \ldots , c_n) \in {\mathcal C}(n)$, we use the Schur algorithm to determine the region of variability $V_\Omega^j (z_0, c ) = \{ \int_0^{z_0} z^{j}(g(z)-g(0))\, d z : g \in {\mathcal F}_\Omega \; \text{with} \; (P^{-1} \circ g) (z) = c_0 +c_1z + \cdots + c_n z^n + \cdots \}$. We  also show that for $z_0 \in {\mathbb D} \backslash \{ 0 \}$ and $c  \in \textrm{Int} \, {\mathcal C}(n) $,  $V_\Omega^j (z_0, c )$ is a convex closed Jordan domain, which we  determine by giving a parametric representation of the boundary curve $\partial V_\Omega^j (z_0, c )$.
\end{abstract}

\thanks{}

\maketitle
\pagestyle{myheadings}
\markboth{Md Firoz Ali, Vasudevarao Allu and Hiroshi Yanagihara}{Variability regions of analytic functions}


\section{Introduction}
Let ${\mathbb C}$ be the complex plane.
For $c \in {\mathbb C}$ and $r > 0$, let
${\mathbb D}(c,r) = \{ z \in {\mathbb C} : |z-c| < r \}$,
and $\overline{\mathbb D}(c,r) = \{ z \in {\mathbb C} : |z-c| \leq r \}$.
In particular, we denote the unit disk by ${\mathbb D} := {\mathbb D}(0,1)$.
Let ${\mathcal A}({\mathbb D})$ be the class of analytic functions
in the unit disk ${\mathbb D}$  with the topology of uniform convergence
on every compact subset of ${\mathbb D}$.
Let $H^\infty ({\mathbb D})$ be the Banach space of
analytic functions $f$ in ${\mathbb D}$ with  norm
$\| f \|_\infty = \sup_{z \in {\mathbb D}} |f(z)|$,
and $H_1^\infty ({\mathbb D})$ be the closed unit ball of $H^\infty ({\mathbb D})$, i.e.,
$H_1^\infty ({\mathbb D}) = \{ \omega \in H^\infty ({\mathbb D}) : \| \omega \|_\infty \leq 1 \}$.
Denote by $\mathcal{S}$, the class of univalent (i.e., one-to-one) functions
$f$ in ${\mathcal A}( {\mathbb D})$, normalized so that $f(0) = f'(0)-1 = 0$.
The subclasses $\mathcal{S}^*$ of starlike, and  ${\mathcal CV}$ of convex functions are respectively  defined by
$$
   \mathcal{S}^* = \{ f \in \mathcal{S} : \; f({\mathbb D})\; \mbox{is starlike with respect to}\; 0 \},
$$
and
$$
   {\mathcal CV} = \{ f \in \mathcal{S} : \; f({\mathbb D})\; \mbox{is convex} \}.
$$
Then ${\mathcal CV} \subset \mathcal{S}^* \subset \mathcal{S}$.
For $f \in {\mathcal A}({\mathbb D})$ with $f(0)=f'(0)-1=0$,
$f \in \mathcal{S}^*$ if, and only if, ${\rm Re} \, ( zf'(z)/f(z) ) > 0$ in ${\mathbb D}$.
Similarly for $f \in {\mathcal A}({\mathbb D})$ with $f(0)=f'(0)-1=0$,
$f \in {\mathcal CV}$ if, and only if, ${\rm Re} \, ( zf''(z)/f'(z) +1) > 0$ in ${\mathbb D}$.
The basic properties of these classes of functions can be found for example in  \cite{Duren-book,Pommerenke-book}.

\bigskip
Let ${\mathcal F}$ be a subclass of ${\mathcal A}({\mathbb D})$
and $z_0 \in {\mathbb D}$. Then upper and lower estimates of the form
$$
 M_1 \leq |f'(z_0)| \leq M_2 ,
 \qquad m_1 \leq \mbox{\rm Arg}\, f'(z_0) \leq m_2
\quad
\mbox{for all } \; f \in {\mathcal F}
$$
are  respectively called a distortion theorem, and
a rotation theorem at $z_0$ for ${\mathcal F}$,
where $M_j$ and $m_j$ ($j=1,2$) are some non-negative constants.
Estimates such as these deal only with absolute values,
or arguments of $f'(z_0)$.
In order to study the complex value $f'(z_0)$ itself,
it is necessary to consider the variability region of
$f'(z_0)$ when $f$ ranges over the class ${\mathcal F}$,
i.e., the set $\{ f'(z_0) : f \in {\mathcal F } \}$.
For example, it is known that
$$
  \{ \log f'(z_0) : f \in {\mathcal CV} \}
  =
  \left\{ \log \frac{1}{(1-z)^2 } : |z| \leq |z_0| \right\} .
$$
For a proof, we refer to \cite[Chapter 2, Exercise 10, 11 and 13]{Duren-book}.
For other examples see \cite{MacGregor-1982,Pinchuk-1968} and the references therein.

\bigskip
Let ${\mathbb H} = \{ w \in {\mathbb C} : \text{\rm Re} \, w > 0 \}$.
For $f \in {\mathcal CV}$, the function $g$ given by $g(z)=1+ z f''(z)/f'(z)$
satisfies $g({\mathbb D}) \subset {\mathbb H}$, i.e.,
$\text{\rm Re} \, g(z) > 0$ in ${\mathbb D}$.
Applying Schwarz's lemma to $(g(z)-1)/(g(z)+1)$
we obtain $|f''(0)| \leq 2$ for $f \in {\mathcal CV}$.
Gronwall \cite{Gronwall-1920}, and independently
Finkelstein \cite{Finkelstein-1967} obtained
the sharp lower and upper estimates
for $|f'(z_0)|$, when $f \in {\mathcal CV}$ satisfies
the additional condition $f''(0) = 2 \lambda$,
where $z_0 \in {\mathbb D}$ and
$\lambda \in \overline{\mathbb D}$ are arbitrarily
preassigned. Let
$$
 \widetilde{V} (z_0, \lambda)
 = \{ \log f'(z_0) : f \in {\mathcal CV} \; \mbox{and} \;
  f''(0) = 2 \lambda \} .
$$
It is easy to see that
$\widetilde{V}(e^{-i \theta}z_0, e^{i \theta }\lambda) = \widetilde{V}(z_0, \lambda )$
for all $\theta \in {\mathbb R}$.
If $| \lambda | =1 $, then by Schwarz's lemma, for $f \in {\mathcal CV}$  the condition
$f''(0) = 2 \lambda $ forces $f(z) \equiv z/(1-\lambda z)$,
and hence $\widetilde{V} (z_0, \lambda) = \{ \log 1/(1-\lambda z_0)^2  \}$.
Thus it suffices to consider the case  $0 \leq \lambda < 1$.
In 2006, one of the present authors \cite{Yanagihra:convex}
obtained the following extension to Gronwall's \cite{Gronwall-1920} result.

\begin{customthm}{A}\label{theorem-A}
For any $z_0 \in {\mathbb D} \backslash \{ 0 \}$ and $0 \leq \lambda < 1$,
the set $\widetilde{V} (z_0, \lambda) $
is a convex closed Jordan domain surrounded by the curve
\begin{eqnarray*}
    && ( - \pi , \pi ] \ni \theta \mapsto
\\
&&
    - \left(
    1 - \frac{\lambda \cos ( \theta /2)}{\sqrt{1 - \lambda^2 \sin^2 ( \theta /2) }}
    \right)
    \log
    \left\{
    1  - \frac{e^{i \theta /2} z_0}
    {i \lambda \sin ( \theta /2) - \sqrt{1 - \lambda^2 \sin^2 ( \theta /2) } }
    \right\}
\\
    & & \; \; \;  \; -
    \left(
    1 + \frac{\lambda \cos ( \theta /2)}{\sqrt{1 - \lambda^2 \sin^2 ( \theta /2) }}
    \right)
    \log
    \left\{
    1  - \frac{e^{i \theta /2} z_0}
    {i \lambda \sin ( \theta /2) + \sqrt{1 - \lambda^2 \sin^2 ( \theta /2) } }
    \right\} .
\end{eqnarray*}
\end{customthm}

In order to prove Theorem \ref{theorem-A}, Yanagihara \cite{Yanagihra:convex} implicitly showed the following.

\begin{customthm}{B}\label{theorem-B}
For any
$z_0 \in {\mathbb D} \backslash \{ 0 \}$ and $0 \leq \lambda < 1$, the variability region
$$
\left\{
     \int_0^{z_0} \frac{g(\zeta )-g(0)}{\zeta} \, d \zeta
     : \; g \in {\mathcal A}({\mathbb D}) \; \mbox{with} \;
     g(0)=1,\; g'(0)= 2 \lambda,\; g({\mathbb D}) \subset {\mathbb H}
  \right\}
$$
is the same convex closed Jordan domain as in Theorem \ref{theorem-A}.
\end{customthm}

Note that putting $g(z) = 1 + zf''(z)/f'(z)$, Theorem \ref{theorem-A} is a direct consequence of
Theorem \ref{theorem-B}. For similar results, we refer to \cite{Ponnusamy-Vasudevarao-2007,Yanagihara:bounded}.

\bigskip
The aim of this paper is to extend
Theorem \ref{theorem-B}, and to refine Theorem \ref{theorem-A}.
Throughout  the paper, we will, unless otherwise stated, assume that $\Omega $ is a convex domain in ${\mathbb C}$ with
$\Omega \not= {\mathbb C}$, and $P$ is a conformal map of
${\mathbb D}$ onto $\Omega$.
\bigskip

Let ${\mathcal F}_\Omega$ be the
class of analytic functions $g$  in ${\mathbb D}$ with
$g( {\mathbb D}) \subset \Omega$. Then the map
${\mathcal F}_\Omega \ni g \mapsto
\omega = P^{-1} \circ g \in
H_1^\infty ({\mathbb D}) \backslash E$
is bijective, where
$E= \{ \omega \in H_1^\infty({\mathbb D}) : \omega(z) \equiv a
\; \text{with} \; |a|=1 \}$.
For $c  = (c_0,c_1, \ldots , c_n ) \in {\mathbb C}^{n+1}$, let
$$
{\mathcal F}_\Omega (c ) =
   \left\{ g \in {\mathcal F}_\Omega : \;
   (P^{-1} \circ g) (z) = c_0 + c_1z+\cdots +c_nz^n + \cdots,  {z\in \mathbb D}  \right\}.
$$
For $f \in {\mathcal A}({\mathbb D})$ with $f(0)=f'(0)-1=0$,
let $g(z) = 1+ z f''(z)/f'(z)$.
Then the coefficients of the Taylor series of $f$ up to order $n+1$
are uniquely determined by the coefficients  of $\omega = P^{-1} \circ g $
up to order $n$, and vice versa. Thus in order to extend Theorem \ref{theorem-B},
we  consider  the following problem.

\begin{problem}\label{the_problem}
Let $n \in {\mathbb N} \cup \{ 0 \}$. For $j=-1, 0, 1, \ldots $,
$z_0 \in {\mathbb D} \backslash \{ 0 \}$
and $c  = (c_0, \ldots , c_n ) \in {\mathbb C}^{n+1}$,
determine the variability region
\begin{equation*}
 V_\Omega^j (z_0, c )
 =
 \left\{
   \int_0^{z_0} \zeta^j \left( g( \zeta ) - g(0) \right) \, d \zeta
   : \; g \in {\mathcal F}_\Omega  (c )
  \right\} .
\end{equation*}
\end{problem}

We note that the coefficient body ${\mathcal C}(n)$ defined by
\begin{align*}
 {\mathcal C}(n):=
&
 \{ c=(c_0,c_1, \ldots , c_n ) \in {\mathbb C}^{n+1}: \; \mbox{there exists $\omega \in H_1^\infty ({\mathbb D})$}\\
&\quad
\mbox{and  } \; \omega (z) = c_0+c_1 z + \cdots + c_n z^n  + \cdots \; \in \; {\mathbb D}
 \},
\end{align*}
is a compact and convex subset of ${\mathbb C}^{n+1}$.
 We note that Schur \cite{Schur-1917,Schur-1986} characterized ${\mathcal C}(n)$ completely.
We also refer to \cite[Chapter I]{Foias-Brazho-book} and
\cite[Chapter 1]{Bakonyi-Constantinescu-1992}, for a detailed treatment.
For $c  = (c_0, \ldots ,c_n) \in {\mathbb C}^{n+1}$,
one can calculate the corresponding Schur parameter
$\gamma = (\gamma_0, \ldots, \gamma_k)$ of $c $,
where $0 \leq k \leq n$, and necessary and sufficient conditions
for $c  \in \textrm{Int} \, {\mathcal C}(n)$, $c  \in \partial {\mathcal C}(n)$
and $c  \not\in {\mathcal C}(n)$ can be described in terms of $\gamma$
(see Section 2).

\bigskip
The organization of this paper is as follows. In Section 2, we summarize the definitions and known facts concerning  the Schur parameter $\gamma$.
We then  state our main theorem and its variant, which gives the solution to Problem \ref{the_problem}.
In Section 3, we state  some lemmas on $p$-valent starlike and convex functions, and
 introduce the Schur polynomials associated with the Schur parameter $\gamma$. We then
 prove several lemmas on the Schur polynomials.
In Section 4, we prove our main theorem, and its variant.

\section{The Schur parameter and statement of results}
For the sake of completeness we state the Carath\'{e}odry interpolation problem and its
solution given by Schur \cite{Schur-1917,Schur-1986}, which plays an important role
 in the  proof of our main theorem.

\begin{problem}[The Carath\'{e}odory interpolation problem]\label{prob:Caratheodory}
Let $n \in {\mathbb N} \cup \{ 0 \}$. For $c=(c_0, \ldots , c_n ) \in {\mathbb C}^{n+1}$,
find necessary and sufficient conditions for the existence of $\omega \in H_1^\infty ({\mathbb D})$
such that $\omega (z)$ has a series expansion of the form
$$
 \omega (z) = c_0+ c_1 z + \cdots + c_n z^n + \cdots, \quad z \in {\mathbb D}.
$$
Furthermore, find an explicit description of all solutions.
\end{problem}

We call $c  = (c_0, \ldots , c_n )$ the Carath\'{e}odory data of length $n+1$.
For $a \in {\mathbb D}$, define $\sigma_a \in \mbox{Aut} ( {\mathbb D})$ by
$$
    \sigma_a (z) = \frac{z+a}{1+ \overline{a}z},
    \quad z \in {\mathbb D} .
$$

We first solve Problem \ref{prob:Caratheodory}, when $n=0$ and $c=(c_0)$.

\bigskip

It follows from the maximum modulus principle that  Problem \ref{prob:Caratheodory}
has no solution if $|c_0| > 1$, and  has a unique solution
$\omega (z) \equiv c_0$, if $|c_0|=1$. So suppose  that  $|c_0| <1$ and
$\omega \in H_1^\infty ({\mathbb D}) $ satisfies $\omega (0) = c_0$.
Then by the maximum modulus principle,
$$
   \omega_1 (z)
   =
   \frac{\omega (z)-c_0}{z(1-\overline{c_0}\omega (z))}
$$
belongs to $H_1^\infty ({\mathbb D})$.
This relation between $\omega$ and $\omega_1$ can be convertible.
Consequently, the set of all solutions $\omega \in H_1^\infty ({\mathbb D})$ is given by
$$
   \omega (z)
   = \sigma_{c_0}(z \omega_1 (z))
   = \frac{z \omega_1 (z) + c_0}
   {1+\overline{c_0}z \omega_1 (z)},
$$
where $\omega_1 \in H_1^\infty ({\mathbb D})$ is arbitrary.

\bigskip
The general solution to Problem \ref{prob:Caratheodory}
can now be obtained from the above consideration and
by recursively applying the following
proposition, which was implicitly proved in
\cite{Schur-1917,Schur-1986} and \cite[Chapter 1]{Foias-Brazho-book}.

\begin{proposition}\label{prop:reduction_to_C_prob}
Let $c  = (c_0, \ldots , c_n ) \in {\mathbb C}^{n+1}$
be a Carath\'{e}odory data of length $n+1$.
\begin{enumerate}[{\rm (1)}]

\item If $|c_0| > 1$, or $|c_0|=1$ with $(c_1, \ldots , c_n) \not= (0, \ldots , 0)$,
then the Carath\'{e}odory problem with  data $c$ has no solution.

\item If  $|c_0|=1$ with $(c_1, \ldots , c_n) = (0, \ldots , 0)$,
then the Carath\'{e}odory problem   data $c$ has the unique solution
$\omega (z) \equiv c_0$.

\item Assume that $|c_0| <1$. If $\omega$ is a solution to the Carath\'{e}odory problem
with  data $c$, then
$$
\omega_1(z) =\frac{\omega(z)-c_0}{z(1- \overline{c_0}\omega (z))}
$$
is a solution to the Carath\'{e}odory problem with  data $c^{(1)}$, where
$c^{(1)} = (c^{(1)}_0, \ldots , c^{(1)}_{n-1})$ is a data of length $n$ defined by
\begin{equation}\label{def:tilde_c}
    c^{(1)}_0 = \frac{c_1}{1-|c_0|^2}, \quad
    c^{(1)}_p
    = \frac{c_{p+1} +
    \overline{c_0} \sum_{\ell=1}^p c^{(1)}_{p-\ell}c_\ell}
    {1-|c_0|^2} \quad (1 \leq p \leq n-1).
\end{equation}
Conversely, if we define $c^{(1)} =  (c^{(1)}_0, \ldots , c^{(1)}_{n-1}) \in {\mathbb C}^{n}$
by \text{\rm (\ref{def:tilde_c})}, and $\omega_1$ is a solution to
the Carath\'{e}odory problem  with  data $c^{(1)}$, then
$$
\omega (z) = \sigma_{c_0} (z \omega_1(z))
 = \frac{z \omega_1 (z)+ c_0}  {1 + \overline{c_0}z \omega_1(z)}
$$
is a solution to the Carath\'{e}odory problem with  data $c=(c_0, \ldots , c_n)$.
\end{enumerate}
\end{proposition}

We note that $c=(c_0, \ldots , c_n)$ and
$(c_0, c^{(1)}) = (c_0, c^{(1)}_0, \ldots , c^{(1)}_{n-1})$
are uniquely determined each other by (\ref{def:tilde_c}).

\bigskip
For a given Carath\'{e}odory data $c =(c_0,\ldots ,c_n) \in {\mathbb C}^{n+1}$,
the Schur parameter $\gamma = (\gamma_0 , \ldots , \gamma_k)$,
$k=0,1, \ldots , n$ is defined as follows.
\bigskip

First, let
$c ^{(0)} = (c_0^{(0)}, \ldots , c_n^{(0)}) = (c_0,\ldots ,c_n)$ and $\gamma_0 = c_0^{(0)}$.
If $| \gamma_0| > 1 $,  we set $k=0$ and $\gamma = ( \gamma_0 )$.
If  $|\gamma_0|=1$,  we set $k=n$, and for $p=1,2,\ldots, n$,
$$
\gamma_p =
\begin{cases}
\infty , & \mbox{if} \; \, c_p^{(0)} \not= 0 \\[4pt]
0, & \mbox{if} \; \, c_p^{(0)} =0 .
\end{cases}
$$
If $|\gamma_0 | < 1$,  we define $c^{(1)} = (c_0^{(1)}, \ldots, c_{n-1}^{(1)}) \in {\mathbb C}^n$ by (\ref{def:tilde_c}).
Now  assume that $\gamma_0, \ldots ,\gamma_{j-1}$ and $c^{(j)} = (c_0^{(j)}, \ldots, c_{n-j}^{(j)})$
are already defined and satisfy $|\gamma_0|<1, \ldots,$ $|\gamma_{j-1}|<1$
for some $j$ with  $1 \leq j \leq n-1$. Then  put $\gamma_j = c_0^{(j)}$.
If $|\gamma_j| > 1$,  put $k=j$ and $\gamma = ( \gamma_0, \ldots , \gamma_j )$.
If $|\gamma_j | =1$,   put $k=n$ and for $p=j+1, \ldots , n$,
$$
\gamma_p =
\begin{cases}
   \infty , & \mbox{if} \; \, c_{p-j}^{(j)} \not= 0 \\[4pt]
   0, & \mbox{if} \; \, c_{p-j}^{(j)} =0 .
\end{cases}
$$
If $|\gamma_j | < 1$,  we define $c ^{(j+1)} = (c_0^{(j+1)}, \ldots, c_{n-j-1}^{(j+1)}) \in {\mathbb C}^{n-j}$
as in Proposition \ref{prop:reduction_to_C_prob}, i.e.,
\begin{equation}\label{eq:c_p_j+1}
 c_0^{(j+1)} = \frac{c_1^{(j)}}{1-|\gamma_j|^2}, \quad
 c_p^{(j+1)} = \frac{c_{p+1}^{(j)}
 + \overline{\gamma_j} \sum_{\ell =1}^p c_{p-\ell}^{(j+1)} c_\ell^{(j)}}{1-|\gamma_j|^2}
    \quad (1 \leq p \leq n- j - 1) .
\end{equation}
Applying this procedure recursively, one obtains the Schur parameter
$\gamma = ( \gamma_0 , \ldots , \gamma_k )$, $k=0, \ldots , n$ of $c =(c_0, \ldots , c_n)$.

\bigskip
When $|\gamma_0|<1, \ldots , |\gamma_j | < 1$, the equations in (\ref{eq:c_p_j+1})
and $\gamma_j = c_0^{(j)}$ show that
$c^{(j)} = (c_0^{(j)}, c_1^{(j)}, \ldots , c_{n-j}^{(j)})$
and $(\gamma_j , c_0^{(j+1)}, \ldots , c_{n-j-1}^{(j+1)})$
are uniquely determined by each other. Thus when $|\gamma_0|<1, \ldots , |\gamma_n| < 1$,
$c=(c_0, \ldots , c_n )= c^{(0)}$ and $\gamma = (\gamma_0, \ldots , \gamma_n )$
are also uniquely determined by each other. For an explicit representation of
$\gamma$ in terms of $c$, we refer to the following results of Schur \cite{Schur-1917,Schur-1986}.

\bigskip

\begin{customthm}{C}[Schur \cite{Schur-1917,Schur-1986}]\label{theorem-C}
For $c=(c_0, \ldots , c_n) \in {\mathbb C}^{n+1}$,
let $\gamma = ( \gamma_0 , \ldots , \gamma_k )$ be the Schur parameter of $c$.

\begin{enumerate}[{\rm (i)}]

\item
\label{interior}
If $k=n$ and $|\gamma_0|<1, \ldots , |\gamma_n|<1$, then all solutions to the
Carath\'{e}odory problem with  data $c$ are given by
\begin{equation*}
\omega (z) =
\sigma_{\gamma_0} (z \sigma_{\gamma_1}(\cdots z \sigma_{\gamma_n}(z \omega^* (z)) \cdots )),
\end{equation*}
where $\omega^* \in H_1^\infty ({\mathbb D})$ is arbitrary.

\item
\label{boundary}
If $k=n$ and $|\gamma_0|<1, \ldots , |\gamma_{i-1}|<1$, $|\gamma_i|=1$, $\gamma_{i+1}= \cdots = \gamma_n=0$
for some $i=0, \ldots , n$, then the Carath\'{e}odory problem with  data $c$ has the unique solution
\begin{equation*}
\omega (z) =
  \sigma_{\gamma_0} (z \sigma_{\gamma_1}(\cdots z \sigma_{\gamma_{i-1}}(\gamma_i z)  \cdots )) .
\end{equation*}

\item
\label{exterior}
If the hypotheses of either {\rm (\ref{interior})} or {\rm (\ref{boundary})} does not hold,
then there is no solution to the Carath\'{e}odory problem with  data $c$.
\end{enumerate}
Furthermore, if the hypothesis of {\rm (\ref{interior})} holds and
\begin{equation}\label{eq:def_omega_k}
  \left\{
 \begin{array}{l}
 \omega_0(z) = \omega(z)  \\[3pt]
 \omega_1(z) = \displaystyle \frac{\omega_0(z)-\omega_0(0)}{z(1-\overline{\omega_0(0) }\omega_0(z))}\\[3pt]
 \vdots \\[3pt]
 \omega_n (z) = \displaystyle  \frac{\omega_{n-1} (z)-\omega_{n-1}(0)}
    {z(1-\overline{\omega_{n-1}(0) }\omega_{n-1}(z))}
 \end{array}
\right.
\end{equation}
then
$$
 \gamma_p = \omega_p(0), \quad \omega_p(z) = c_0^{(p)}+c_1^{(p)}z+ \cdots + c_{n-p}^{(p)}z^{n-p} + \cdots
$$
holds for $p=0,1, \ldots , n$.
\end{customthm}

For a detailed proof of Theorem \ref{theorem-C}, we refer to \cite[Chapter 1]{Foias-Brazho-book}.
It is not difficult to see that for a given Carath\'{e}odory data $c $,
the hypotheses of (\ref{interior}), (\ref{boundary}) and (\ref{exterior})
in Theorem \ref{theorem-C} are respectively equivalent to $c  \in \text{Int} \, {\mathcal C}(n)$,
$c  \in \partial {\mathcal C}(n)$ and $c  \not\in {\mathcal C}(n)$.

\bigskip

Let $\Omega$ be a convex domain with $\Omega \not= {\mathbb C}$, and $P$ be a conformal map of ${\mathbb D}$ onto $\Omega$. When $c \not\in {\mathcal C}(n)$ or $c \in \partial {\mathcal C}(n)$, Theorem \ref{theorem-C}  gives the following simple solution to Problem \ref{the_problem}.

\bigskip

\begin{theorem}
Let  $c  =(c_0, \ldots , c_n) \in {\mathbb C}^{n+1}$ be a Carath\'{e}odory data.
If $c  \not\in {\mathcal C}(n)$ i.e., the hypothesis of {\rm (\ref{exterior})} in Theorem \ref{theorem-C} holds,
then $V_\Omega^j (z_0,c ) = \emptyset$. If $c  \in \partial {\mathcal C}(n)$ i.e.,
the hypothesis of {\rm (\ref{boundary})} in Theorem \ref{theorem-C} holds,
then $V_\Omega^j (z_0,c )$ reduces to a set consisting of a single point $w_0$,
where
$$
w_0 = \int_0^{z_0} \zeta^j
 \{ P( \sigma_{\gamma_0} ( \zeta \sigma_{\gamma_1}(\cdots \zeta \sigma_{\gamma_{i-1}} (\gamma_i \zeta  ) \cdots )))- P(c_0) \} \, d \zeta,
$$
and $\gamma =( \gamma_0, \ldots , \gamma_i, 0, \ldots , 0 )$ is  the Schur parameter of $c$.
\end{theorem}

Now we introduce a family of functions, which are extremal for Problem \ref{the_problem}
in the case $c  \in \text{Int} \, {\mathcal C}(n)$.
\bigskip

For $\varepsilon \in \overline{\mathbb D}$
and  Schur parameter $\gamma =( \gamma_0,\ldots , \gamma_n)$ of $c  \in \text{Int} \, {\mathcal C}(n)$, let
\begin{align}
\omega_{\gamma , \varepsilon }(z)
    =&  \sigma_{\gamma_0} ( z \sigma_{\gamma_1} ( \cdots z \sigma_{\gamma_{n}} ( \varepsilon z) \cdots )), \quad z \in {\mathbb D} ,
\label{def:extremal_omega}
\\
Q_{\gamma , j} (z, \varepsilon )
    =& \int_0^z \zeta^j \{ P(\omega_{\gamma,\varepsilon} (\zeta )) - P(c_0) \}\, d \zeta , \quad z \in {\mathbb D} \; \text{and} \;
    \varepsilon \in \overline{\mathbb D}.
\label{def:extremal_Q}
\end{align}
Then $\omega_{\gamma , \varepsilon } \in H_1^\infty ({\mathbb D})$ is a solution to the Carath\'{e}odory problem with the data $c$, i.e.,
$$
\omega_{\gamma , \varepsilon }(z) = c_0+c_1z+\cdots + c_nz^n + \cdots .
$$
In particular, we have $\omega_{\gamma , \varepsilon }(0)=c_0$. We note that for each fixed $\varepsilon \in \overline{\mathbb D}$, $\omega_{\gamma,\varepsilon} (z)$ and $Q_{\gamma , j} (z, \varepsilon )$ are analytic functions of $z \in {\mathbb D}$, and for each fixed $z \in {\mathbb D}$, $\omega_{\gamma,\varepsilon} (z)$ and $Q_{\gamma , j} (z, \varepsilon )$ are analytic functions of $\varepsilon \in \overline{\mathbb D}$.
When $\varepsilon \in \partial {\mathbb D}$, $\omega_{\gamma,\varepsilon} (z)$ is a finite Blaschke product of $z$. Indeed, it follows from (\ref{def:extremal_omega}) that $\omega_{\gamma , \varepsilon } (z)$ is a rational function of $z$, which is analytic on $\overline{\mathbb D}$. If $|\varepsilon | =1$, then again by (\ref{def:extremal_omega}) it is easy to see that $\omega_{\gamma , \varepsilon } $ maps $\partial {\mathbb D}$ into $\partial {\mathbb D}$. Thus $\omega_{\gamma , \varepsilon }$ is a finite Blaschke product of $z$.
\bigskip

We next state the following,  which is the main result of this paper.

\begin{theorem}\label{thm:Main_theorem}
Let $n \in {\mathbb N} \cup \{ 0 \}$, $j \in \{-1,0, 1,2 , \ldots \}$,
$c  = (c_0, \ldots , c_n ) \in \text{\rm Int} \, {\mathcal C}(n)$
and $\gamma =(\gamma_0, \ldots , \gamma_n )$ be the Schur parameter of $c$.
Then for each fixed $z_0 \in {\mathbb D} \backslash \{0 \}$, $Q_{\gamma , j}(z_0, \varepsilon )$
is a convex univalent function of $\varepsilon \in \overline{\mathbb D}$ and
$$
V_\Omega^j (z_0,c) =   Q_{\gamma , j}(z_0, \overline{\mathbb D} )
  := \{ Q_{\gamma , j}(z_0, \varepsilon ) : \varepsilon \in \overline{\mathbb D} \} .
$$
Furthermore
$$
\int_0^{z_0} \zeta^j \{ g(\zeta ) - g(0) \} \, d \zeta = Q_{\gamma , j}(z_0, \varepsilon )
$$
for some $g \in {\mathcal F}_\Omega (c )$ and $\varepsilon \in \partial {\mathbb D}$
if, and only if, $g (z) \equiv P( \omega_{\gamma , \varepsilon } (z ))$.
\end{theorem}

%
%
%
%
%
%

\section{Preliminaries}
First we note the following (which is elementary).

\begin{lemma}\label{lemma:asymptotic_series}
Let  $n \in {\mathbb N} \cup \{ 0 \}$ and $g_0$, $g_1 \in {\mathcal A}({\mathbb D})$ satisfying
$g_1(z) -g_0(z) = {\mathcal{O}}(z^{n+1})$ as $z \rightarrow 0$. Then for any analytic function $\varphi$ defined in
a neighborhood of $g_0(0)=g_1(0)$, $\varphi (g_0(z)) - \varphi (g_1(z)) = { \mathcal{O}}(z^{n+1})$ as $z \rightarrow 0$.
\end{lemma}

We next give some  simple lemmas concerning multivalent starlike and convex functions.
For a positive integer $p$, we write $(\mathcal{S}^*)^p:=\{ f_0^p:f_0\in\mathcal{S}^* \}$.

\begin{lemma}\label{lemma:p-valent_starlikeness}
Let $f \in {\mathcal A}( {\mathbb D})$
with $f(z) = z^p  + \cdots $ for ${z\in \mathbb D}$. Then $f \in (\mathcal{S}^*)^p$, i.e., there exists
$f_0 \in \mathcal{S}^*$ with $f=f_0^p$ if, and only if,
\begin{equation}\label{ineq:analytic_condition_of_p-valent_starlikeness}
{\rm Re} \, \left\{  \frac{zf'(z)}{f(z)} \right\} > 0, \quad  \; {z\in \mathbb D} .
\end{equation}
\end{lemma}

\begin{proof}
When $p=1$,  the lemma is well known (see \cite[Theorem 2.10]{Duren-book}).
For general $p$, if there exists $f_0 \in \mathcal{S}^*$ with $f=f_0^p$, then by the identity $zf'(z)/f(z)=pz f_0'(z)/f_0(z)$
we have $\text{\rm Re} \, \{ zf'(z)/f(z) \} > 0$.

\bigskip
Inequality (\ref{ineq:analytic_condition_of_p-valent_starlikeness}) ensures that $f$  has no zeros in ${\mathbb D} \backslash \{ 0 \}$. Thus there exists $f_0 \in {\mathcal A}({\mathbb D})$ with $f_0^p = f$ and $f_0'(0)=1$. Again from the identity $zf'(z)/f(z)=pz f_0'(z)/f_0(z)$ it follows that $\text{\rm Re} \, \{ z f_0'(z)/f_0(z) \} > 0$ in ${\mathbb D}$. Thus $f_0 \in \mathcal{S}^*$.
\end{proof}

\begin{lemma}\label{lemma:p-valent_convexity}
Let $f \in {\mathcal A}( {\mathbb D})$
with $f(z) = z^p  + \cdots $ for ${z\in \mathbb D}$. If $f$ satisfies
\begin{equation*}
  {\rm Re} \, \left\{1+ \frac{zf''(z)}{f'(z)} \right\} > 0, \quad  \; {z\in \mathbb D},
\end{equation*}
then  there exists $f_0 \in \mathcal{S}^*$ with $f=f_0^p$.
\end{lemma}

\begin{proof}
Let $q \in {\mathbb N} \cup \{ 0 \}$, and $A$ and  $B$ be analytic functions in ${\mathbb D}$
such that $A(z)=az^q + \cdots $ and $B(z)=bz^q + \cdots $ in ${\mathbb D}$ with $a, b \not= 0 $.
Suppose that $a^{-1} A \in (\mathcal{S}^*)^q$, or $b^{-1} B \in (\mathcal{S}^*)^q$. Then Libera \cite{Libera-1965} showed that
$\textrm{Re} \, \{ A'(z)/B'(z) \} >0$ in ${\mathbb D}$ implies $\textrm{Re} \, \{ A(z)/B(z) \} >0$ in ${\mathbb D}$.
Now note that $1+zf''(z)/f'(z) = (zf'(z))'/f'(z)$. Let $g(z) = p^{-1}zf'(z)$. Since $\text{\rm Re} \, ( zg'(z)/g(z) )
= \mbox{\rm Re} ( 1+zf''(z)/f'(z) )>0$, it follows from  Lemma \ref{lemma:p-valent_starlikeness} that $g \in (\mathcal{S}^*)^p$.
Thus, with $A(z)=zf'(z)$ and $B(z)=f(z) $ Libera's result, shows that  $\text{\rm Re} \, ( zf'(z)/f(z) ) >0$ for  ${z\in \mathbb D}$, and so
the lemma now follows from Lemma \ref{lemma:p-valent_starlikeness}.
\end{proof}

\bigskip

A well known theorem due to Robertson \cite{Robertson-1936} states that each $f \in {\mathcal CV}$ can be approximated by a sequence of conformal maps of
${\mathbb D} $ onto convex polygons, uniformly on any compact subset of ${\mathbb D}$ as follows.

\bigskip

\begin{lemma}\cite{Robertson-1936}\label{lemma:approx_by_polygon_mapping}
For any function $f \in {\mathcal CV}$, $\eta > 0$ and a compact subset $E$ of ${\mathbb D}$,
there exists $h \in {\mathcal A}({\mathbb D})$ with $\sup_{z \in E} |f'(z)-h(z) | < \eta$, such that
$h(z)$ can be written in the form
$$
h(z) = \frac{1} {\prod_{i=1}^m (1-\eta_i z)^{\beta_i}},
$$
where $\eta_i \in \partial {\mathbb D}$, $0<\beta_i \leq 2$ $(i=1,\ldots , m)$, and $\sum_{i=1}^m \beta_i = 2$.
\end{lemma}

The following brief summary of Schur polynomials will assist in the proof of our main theorem. For more details we refer to \cite{Bakonyi-Constantinescu-1992}.

\bigskip

Let $c  =(c_0, \ldots ,c_n ) \in \textrm{Int} \, {\mathcal C}(n)$ and $\gamma =( \gamma_0, \ldots , \gamma_n )$
be the Schur parameter of $c $. Note that $c  =(c_0, \ldots ,c_n ) \in \textrm{Int} \, {\mathcal C}(n)$ forces
$|\gamma_0| < 1, \ldots , |\gamma_n| < 1$. Suppose that $\omega \in H_1^\infty ({\mathbb D})$
satisfies $\omega (z) = c_0 + \ldots + c_n z^n + \cdots $, and define $\omega_0, \omega_1, \ldots , \omega_n
\in H_1^\infty ({\mathbb D})$ by (\ref{eq:def_omega_k}) and $\omega_{n+1} \in H_1^\infty ({\mathbb D})$ by
$$
\omega_{n+1} = \frac{\omega_n(z) - \omega_n(0)}{z(1-\overline{\omega_n(0)} \omega_n(z))} .
$$
Then since $\gamma_k = \omega_k(0)$, $k=0, \ldots , n$, we have
\begin{equation}\label{eq:reverse_Algorithm_of_Schur}
\omega (z ) =  \omega_0 (z) \quad\mbox{and}\quad
\omega_k (z ) = \frac{z \omega_{k+1} (z)+ \gamma_k }{1+ \overline{\gamma_k }z \omega_{k+1} (z)},\quad (k=0,1,\ldots,n).
\end{equation}
We now define sequences of polynomials recursively by
\begin{equation}\label{eq:initial_condition}
 A_0 (z) =\overline{\gamma_0} , \; B_0 (z) =1, \; \widetilde{A}_0 (z) = 1  , \; \widetilde{B}_0 (z) = \gamma_0  ,
\end{equation}
and
\begin{equation}\label{eq:recurrence_formula}
 \begin{pmatrix}
  A_{k+1} (z) & \widetilde{A}_{k+1}(z) \\
  B_{k+1} (z) & \widetilde{B}_{k+1} (z) \\
 \end{pmatrix}
 =
 \begin{pmatrix}
  z & \overline{\gamma_{k+1}} \\
  \gamma_{k+1}z & 1 \\
 \end{pmatrix}
 \begin{pmatrix}
  A_k (z) & \widetilde{A}_k (z) \\
  B_k (z) & \widetilde{B}_k (z) \\
 \end{pmatrix} ,
\end{equation}
where $k=0,\ldots , n-1 $. Then from (\ref{eq:reverse_Algorithm_of_Schur}), we have
\begin{equation}\label{eq:recurrence_formula_for_omega_k}
      \omega (z) =
      \frac{z \widetilde{A}_k(z) \omega_{k+1} (z) + \widetilde{B}_k(z)}
      {zA_k(z) \omega_{k+1} (z) + B_k(z)}
      , \quad
      k=0,1, \ldots , n .
\end{equation}
The polynomials $A_k (z), B_k (z), \widetilde{A}_k (z)$ and $\widetilde{B}_k(z)$ ($k=0,1, \ldots , n$)
are of degree at most $k$, and are called the Schur polynomials associated with $\gamma$.
For convenience, we put $\omega^*(z) = \omega_{n+1}(z)$. Then by Theorem \ref{theorem-C} and
(\ref{eq:recurrence_formula_for_omega_k}) we obtain the following.

\begin{lemma}\label{lemma:Schur_polynomial_and_Caratheodoy_problem}
Let $c  = (c_0, \ldots , c_n) \in \text{\rm Int}\, {\mathcal C}(n)$, and $\gamma=( \gamma_0, \ldots , \gamma_n )$ be the Schur parameter of $c $.
Then for any $\omega \in H^\infty_1 ({\mathbb D})$ with $\omega (z) = c_0 + \cdots + c_n z^n + \cdots $ in ${\mathbb D}$, there exists a unique $\omega^* \in H^\infty_1 ({\mathbb D})$ such that
\begin{equation}\label{eq:SChur_representation}
\omega (z) = \frac{z\widetilde{A}_n(z) \omega^*(z)+ \widetilde{B}_n(z)}{zA_n(z)\omega^*(z) + B_n(z)} .
\end{equation}
Conversely, for any $\omega^* \in H^\infty_1 ({\mathbb D})$, if  $\omega$ is given by $(\ref{eq:SChur_representation})$, then $\omega$ satisfies $\omega (z) = c_0 + \cdots + c_n z^n + \cdots $ in ${\mathbb D}$. In particular, the function $\omega_{\gamma , \varepsilon }$ defined by (\ref{def:extremal_omega}) can be written as
\begin{equation}\label{eq:extremal_omega_and_schur_polynomials}
\omega_{\gamma , \varepsilon } (z)
   = \frac{ \varepsilon z\widetilde{A}_n(z)+ \widetilde{B}_n(z)}{ \varepsilon z A_n(z) + B_n(z)},
\end{equation}
and $P\circ \omega_{\gamma,\varepsilon } \in {\mathcal F}_\Omega (c )$.
\end{lemma}

We next give some  important properties of the Schur polynomials.
\bigskip

From (\ref{eq:initial_condition}) and
(\ref{eq:recurrence_formula}) it easily follows that $\widetilde{A}_k(z) $ is a monic polynomial of
degree $k$. Also
\begin{equation}\label{eq:B_of_zero}
B_k(0) = 1 , \quad \text{and} \quad  \widetilde{B}_k(0) = \gamma_0
\end{equation}
for $k=0,1, \ldots , n$ and the degrees of $A_k(z), B_k(z)$ and $\widetilde{B}_k(z)$ are at most $k$.

\begin{lemma}\label{lemma:relation_between_Moebius_coefficients}
For $k=0, \ldots , n$,
\begin{equation}\label{relation_between_Moebius_coefficients}
  \widetilde{A}_k (z) = z^k \overline{B_k (1/\overline{z})} ,
   \quad
   \widetilde{B}_k (z) = z^k \overline{A_k (1/\overline{z})} .
\end{equation}
\end{lemma}

\begin{proof}
When $k=0$, (\ref{relation_between_Moebius_coefficients}) directly follows form (\ref{eq:initial_condition}).
Assume that (\ref{relation_between_Moebius_coefficients}) holds for $k \geq 0$,
then by (\ref{eq:recurrence_formula}) we have
\begin{eqnarray*}
\widetilde{A}_{k+1} (z)
     &=&
     z \widetilde{A}_k (z) + \overline{\gamma_{k+1}} \widetilde{B}_k (z)\\
     &=&
     z^{k+1} \overline{B_k(1/\overline{z})} + \overline{\gamma_{k+1}} z^k  \overline{A_k(1/ \overline{z})}\\
     &=&
     z^{k+1}
     \left\{\overline{\gamma_{k+1}(1/ \overline{z}) A_k(1/ \overline{z}) + B_k(1/ \overline{z})}\right\}
     =
     z^{k+1} \overline{B_{k+1} (1/ \overline{z})} .
\end{eqnarray*}
Similarly  $\widetilde{B}_{k+1} (z) = z^{k+1} \overline{A_{k+1} (1/ \overline{z})}$.
\end{proof}

\begin{lemma}\label{lemma:determinat_of_Mobius_transformation}
For $k=0, \ldots , n$,
\begin{equation}\label{eq:determinat_of_Mobius_transformation}
  \widetilde{A}_k (z) B_k(z) - A_k(z)  \widetilde{B}_k(z) = z^k \prod_{\ell=0}^k (1-|\gamma_\ell|^2) .
\end{equation}
\end{lemma}

\begin{proof}
We use induction on $k$. When $k=0$, (\ref{eq:determinat_of_Mobius_transformation})
 follows directly from (\ref{eq:initial_condition}). Assume
(\ref{eq:determinat_of_Mobius_transformation}) holds for $k \geq 0$,
then by (\ref{eq:recurrence_formula}) we have
\begin{align*}
& \widetilde{A}_{k+1} (z) B_{k+1}(z)- A_{k+1}(z) \widetilde{B}_{k+1}(z) \\
=& \{ z \widetilde{A}_k (z) + \overline{\gamma_{k+1}} \widetilde{B}_k(z) \} \{ \gamma_{k+1} z A_k(z) + B_k(z) \}\\
& - \{ z A_k(z) + \overline{\gamma_{k+1}} B_k(z) \} \{ \gamma_{k+1} z \widetilde{A}_k(z) + \widetilde{B}_k(z) \}\\
=& z(1-|\gamma_{k+1}|^2 )\{\widetilde{A}_k (z) B_k(z) - A_k(z)  \widetilde{B}_k(z) \}
  = z^{k+1} \prod_{\ell=0}^{k+1} (1-| \gamma_\ell |^2) .
\end{align*}
\end{proof}

\begin{lemma}\label{lemma:denominator_has_no_zeros}
For $k=0, \ldots , n $, the inequality
\begin{equation}\label{ineq:denominator_has_no_zeros}
|B_k(z)|^2 - | A_k (z)|^2
\geq \prod_{\ell=0}^k (1-|\gamma_\ell|^2)
\end{equation}
holds for $z\in \overline{\mathbb D}$.
\end{lemma}

\begin{proof}
When $k=0$,  it follows from (\ref{eq:initial_condition}) that $|B_0(z)|^2-|A_0(z)|^2 = 1-|\gamma_0|^2$.
Assume (\ref{ineq:denominator_has_no_zeros}) holds for $k \geq 0$,
then for $|z| \leq 1$,  and  (\ref{eq:recurrence_formula}) we have
\begin{align*}
   & |B_{k+1}(z)|^2  - |A_{k+1}(z)|^2
\\
   =&
   |\gamma_{k+1} z A_k(z) + B_k(z) |^2
   -
   |z A_k(z) + \overline{\gamma_{k+1} }B_k(z)|^2
\\
  =&
   (1-|\gamma_{k+1}|^2) (|B_k(z)|^2  - |z|^2 |A_k(z)|^2 )
\\
  \geq&
   (1-|\gamma_{k+1}|^2) (|B_k(z)|^2  - |A_k(z)|^2 )
\\
  \geq&
   (1-|\gamma_{k+1}|^2) \prod_{\ell=0}^k (1-|\gamma_\ell |^2 )
   = \prod_{\ell=0}^{k+1} (1-|\gamma_\ell |^2 ) .
\end{align*}
\end{proof}

\begin{lemma}\label{lemma: tilde_omega_is_hol_in_D}
For $k=0, \ldots , n$, the inequality
\begin{equation*}\label{ineq: tilde_omega_is_hol_in_D}
|\widetilde{B}_k (z)| < |B_k(z)|
\end{equation*}
holds for $z\in \overline{\mathbb D}$.
\end{lemma}

\begin{proof}
By Lemma \ref{lemma:denominator_has_no_zeros}, the function $B_k(z)$ has no zeros on $\overline{\mathbb D}$.
Hence  $\widetilde{B}_k(z)/B_k(z)$ is analytic on $\overline{\mathbb D}$. For $|z|=1$,  using
Lemmas \ref{lemma:relation_between_Moebius_coefficients} and \ref{lemma:denominator_has_no_zeros} we have
\begin{align*}
   |B_k(z)|^2 - |\widetilde{B}_k(z)|^2
    &=
    |B_k(z)|^2 - |z^k \overline{A_k (1/\overline{z})}|^2
\\
    &= |B_k(z)|^2 - |A_k(z)|^2
    \geq \prod_{\ell=0}^k (1-|\gamma_\ell |^2 ) > 0 .
\end{align*}
Thus we have $|\widetilde{B}_k(z)/B_k(z)| < 1$ on $\partial {\mathbb D}$,
and hence by the maximum modulus principle for analytic functions,
$|\widetilde{B}_k(z)/B_k(z)| < 1$ holds on $\overline{\mathbb D}$.
\end{proof}

%

\section{Proof of the Main Theorem}
First we show that $V_\Omega^j (z_0, c ) $ is a compact and convex subset of ${\mathbb C}$.

\begin{proposition}\label{prop:compact_and_convex}
For $c =(c_0,\ldots ,c_n) \in \text{\rm Int}\,{\mathcal C}(n)$, the class
${\mathcal F}_\Omega (c)$ is a compact and convex subset of ${\mathcal A}({\mathbb D})$.
\end{proposition}

\begin{proof}
For $g \in {\mathcal F}_\Omega (c )$, by Schwarz's lemma we have
$$
   \left|   \frac{P^{-1} (g(z)) -c_0}{1- \overline{c_0}P^{-1} (g(z))}   \right|
   \leq  |z|, \quad z \in {\mathbb D} .
$$
This implies $P^{-1}(g(z)) \in \overline{\Delta} (c_0, r) = \{ w \in {\mathbb C} : |w-c_0|/|1-\overline{c_0} w | \leq r \} (\subset {\mathbb D})$
for $|z| \leq r < 1$. Thus $g(z) \in P( \overline{\Delta} (c_0, r) )$ for any
$g \in {\mathcal F}_\Omega (c)$, and $|z| \leq r < 1$. Therefore ${\mathcal F}_\Omega (c)$ is
locally uniformly bounded, and hence by Montel's theorem forms a normal family.

\bigskip
We next show that ${\mathcal F}_\Omega (c)$ is closed.
\bigskip

Let $g_k \in {\mathcal F}_\Omega (c)$, $k \in {\mathbb N}$
and $g_0 \in {\mathcal A}({\mathbb D})$ such that $g_k \rightarrow g_0$ locally uniformly in ${\mathbb D}$
as $k \rightarrow \infty$. Then
$$
\frac{1}{j!}(P^{-1} \circ g_0)^{(j)}(0) = \frac{1}{j!}\lim_{k \rightarrow \infty} (P^{-1} \circ g_n)^{(j)}(0) = c_j
\quad\mbox{for}\quad j=0,\ldots , n.
$$
Thus $(P^{-1} \circ g_0)(z) = c_0 + c_1 z +\cdots + c_n z^n + \cdots $ for ${z\in \mathbb D}$.
Similarly, it follows that $g_0 ({\mathbb D}) \subset \overline{\Omega}$.
Now suppose that $g_0({\mathbb D}) \backslash \Omega \not= \emptyset$.
Then there exists $z^* \in {\mathbb D}$ such that $w^* = g_0(z^*) \in \partial \Omega$.
Since $g_0 (0) = P(c_0) \not= w^* =g_0 (z^*)$, $g_0$ is a non-constant analytic function and hence  is an open map.
Thus $g_0({\mathbb D})$ is a neighborhood of $w^* \in \partial \Omega$.
Since there exists a support line of the convex set $\Omega$ through $w^*$,
the neighborhood $g_0 ({\mathbb D})$ of $w^*$ contains an exterior point of $\Omega$.
This contradicts $g_0 ({\mathbb D}) \subset \overline{\Omega}$.
Hence $g_0 ({\mathbb D}) \subset \Omega$ and $g_0 \in {\mathcal F}_\Omega (c)$,
and so ${\mathcal F}_\Omega (c)$ is closed in ${\mathcal A}({\mathbb D})$.
Since ${\mathcal F}_\Omega (c)$ is normal and closed in the metric space ${\mathcal A}({\mathbb D})$,
it is therefore a compact subset of ${\mathcal A}({\mathbb D})$.

\bigskip
Next suppose $g_0$, $g_1 \in {\mathcal F}_\Omega (c )$. Then $P^{-1}(g_0(z)) - P^{-1}(g_1(z))
= \mathcal{O}(z^{n+1})$ as $z \rightarrow 0$. Hence by Lemma \ref{lemma:asymptotic_series}, $g_0(z)-g_1(z) = \mathcal{O} (z^{n+1})$ as $z \rightarrow 0$. Thus for any $t \in [0,1]$,
$g_t(z) = (1-t)g_0(z)+tg_1(z)$ satisfies $g_t(z) - g_0(z) = \mathcal{O} (z^{n+1})$ as $z \rightarrow 0$.
Hence  again using Lemma \ref{lemma:asymptotic_series}, it follows that $P^{-1}(g_t(z)) - P^{-1}(g_0(z)) = \mathcal{O} (z^{n+1})$,
as $z \rightarrow 0$. This  shows that $P^{-1} (g_t(z)) = c_0+ \cdots  + c_n z^n + \cdots $
for ${z\in \mathbb D}$. Since $\Omega $ is convex, $g_t (z) = (1-t)g_0 (z) + t g_1 (z) \in \Omega$ for all $z \in {\mathbb D}$.
Thus $g_t({\mathbb D}) \subset \Omega$, and hence $g_t \in {\mathcal F}_\Omega (c )$.
Therefore  ${\mathcal F}_\Omega (c )$ is convex.
\end{proof}

\begin{corollary}\label{cor:compact_and_convex}
For $c =(c_0,\ldots ,c_n) \in \text{\rm Int} \, {\mathcal C}(n)$, the set
$V_\Omega^j (z_0, c )$ is a compact and convex subset of ${\mathbb C}$.
\end{corollary}

\begin{proof}
Since the functional ${\mathcal A}({\mathbb D}) \ni g \mapsto \int_0^{z_0} \zeta^j (g(\zeta ) - g(0)) \, d \zeta $
is linear and continuous on ${\mathcal A}({\mathbb D})$, and  $V_\Omega^j(z_0, c )$
is the image of the compact and convex subset ${\mathcal F}_\Omega (c)$ of ${\mathcal A}({\mathbb D})$
under this functional, $V_\Omega^j(z_0, c )$ is also a compact and convex subset of ${\mathbb C}$.
\end{proof}

\begin{proposition}\label{prop:interior_point}
Let $c =(c_0,\ldots ,c_n) \in {\rm Int} \, {\mathcal C}(n)$ and $z_0 \in {\mathbb D} \backslash \{ 0 \}$.
Then $Q_{\gamma,j}(z_0, 0) = \int_0^{z_0} \zeta^j \{ P( \omega_{\gamma, 0} (\zeta ) - P(c_0) \} \,  d \zeta $
is  an interior point of the set $V_\Omega^j(z_0, c )$, where $\gamma=(\gamma_0,\ldots , \gamma_n)$ is the Schur parameter of $c $.
\end{proposition}

\begin{proof}
For fixed $z_0 \in {\mathbb D}$,  note that $Q_{\gamma,j}(z_0, \varepsilon )$ defined by (\ref{def:extremal_Q})
is an analytic function of $\varepsilon \in \overline{\mathbb D}$. Since a non-constant analytic function is an open map,
in order to prove the proposition it suffices to show that $Q_{\gamma,j}(z_0, \varepsilon )$ is a non-constant
function of $\varepsilon$.

\bigskip

Let
$$
   \psi(z) =
   \left.
   \frac{\partial }{\partial \varepsilon} \left\{ Q_{\gamma,j}(z, \varepsilon ) \right\}
   \right|_{\varepsilon=0} .
$$
Then the problem reduces to showing  that $\psi(z) \not= 0$ for $z \in {\mathbb D} \backslash \{ 0 \}$.
By (\ref{eq:extremal_omega_and_schur_polynomials}) and Lemma \ref{lemma:determinat_of_Mobius_transformation}
we have
\begin{align*}
\psi(z)
&= \int_0^z \zeta^j P' \left( \frac{\widetilde{B}_n( \zeta )}{B_n(\zeta )} \right)
    \frac{\zeta \left\{ \widetilde{A}_n( \zeta ) B_n( \zeta ) - \widetilde{B}_n( \zeta ) A_n( \zeta ) \right\}}{B_n( \zeta )^2} \, d\zeta
\\
&= \prod_{k=0}^n (1-|\gamma_k|^2) \int_0^z \zeta^{n+j+1} P' \left( \frac{\widetilde{B}_n(\zeta )}{B_n( \zeta )} \right)
  \frac{d  \zeta}{B_n( \zeta )^2} .
\end{align*}
We note that $\psi$ has  a Taylor series representation of the form $\psi(z)=az^{n+j+2}+\cdots$,
with $a=(n+j+2)^{-1} \prod_{k=0}^n (1-|\gamma_k|^2) P'(c_0 )$. The assertion will be proved
provided we can  show that $\mbox{\rm Re}\, ( z\psi''(z)/\psi'(z)) \geq 0$, since in this case,
  Lemmas \ref{lemma:p-valent_convexity} and (\ref{eq:B_of_zero}) imply
that there exists a starlike univalent function $\psi_0 \in \mathcal{S}^*$ satisfying
$\psi(z) = a \psi_0(z)^{n+j+2}$, and so $\psi(z) $  has no zeros in ${\mathbb D} \backslash \{ 0 \}$.

\bigskip
In order to show $\text{\rm Re}\, ( z\psi''(z)/\psi'(z)) \geq 0$, by Lemma \ref{lemma:approx_by_polygon_mapping}
we may assume without loss of generality that
$$
 P'(z) = \frac{P'(0)}{\prod_{i=1}^m ( 1 - \eta_i z )^{\beta_i}},
$$
where $\eta_i \in \partial {\mathbb D}$, $0< \beta_i \leq 2$ $(i=1,2, \ldots , m)$, and $\sum_{i=1}^m \beta_i = 2$.
Under this assumption
$$
\psi(z) = P'(0) \prod_{k=0}^m (1-| \gamma_k |^2)
  \int_0^z \frac{\zeta^{n+j+1}} {\prod_{i=1}^m (B_n( \zeta ) - \eta_i \widetilde{B}_n( \zeta ))^{\beta_i}}\, d\zeta
$$
and so
$$
z\frac{\psi''(z)}{\psi'(z)}
  = n+j+1 - \sum_{i=1}^m \beta_i z \frac{d}{dz} \log \left\{ B_n (z)- \eta_i \widetilde{B}_n(z) \right\} .
$$
For $i=1, \ldots ,m$, let $z_{i1}, \ldots ,z_{ip_i}$ $(0 \leq p_i \leq n)$ be zeros of
$B_n(z) - \eta_i\widetilde{B}_n (z) $. Then by Lemma \ref{lemma: tilde_omega_is_hol_in_D},
we have $|z_{i \ell} | >1 $ for all $i =1, \ldots m$ and $\ell =1, \ldots , p_i$.
Since $B_n (0) =1 $ and $\widetilde{B}_n (0) = \gamma_0 $, it follows that
$B_n(z) - \eta_i \widetilde{B}_n (z) = (1- \eta_i \gamma_0 ) \prod_{\ell=1}^{p_i} (1-z/z_{i\ell})$.
Thus using the identity $w/(1-w) = 2^{-1}\{(1+w)/(1-w))-1\}$ we have
\begin{align*}
 \text{\rm Re} \left\{ z\frac{\psi''(z)}{\psi'(z)}   \right\}
  =&
  n+ j+1  + \text{\rm Re} \left\{ \sum_{i=1}^m  \beta_i \sum_{\ell=1}^{p_i} \frac{z/z_{i\ell}}{1- z/z_{i\ell}}  \right\}\\
  =&
   n+j+1 - \sum_{i=1}^m \frac{p_i \beta_i}{2}  +
   \sum_{i=1}^m \beta_i \sum_{\ell=1}^{p_i} \text{\rm Re} \left\{ \frac{1+ z/z_{i\ell}}{1- z/z_{i\ell}} \right\}\\
  \geq&
   j+1 + \sum_{i=1}^m \beta_i \sum_{\ell=1}^{p_i}
    \text{\rm Re} \left\{\frac{1+ z/z_{i\ell}}{1- z/z_{i\ell}} \right\} \geq 0 .
\end{align*}
\end{proof}

\begin{proposition}\label{prop:boundary_uniqueness}
Let $c  =(c_0,\ldots ,c_n) \in {\rm Int} \, {\mathcal C}(n)$ and $z_0 \in {\mathbb D} \backslash \{ 0 \}$.
Then
$$
  Q_{\gamma, j}(z_0, \varepsilon ) \in \partial V_\Omega^j(z_0, c )
$$
holds for all $\varepsilon \in \partial {\mathbb D}$, where $\gamma$ is the Schur parameter of $c $.
Furthermore,
$$
 \int_0^{z_0} \zeta^j \{g(\zeta)-g(0) \} \, d \zeta = Q_{\gamma, j}(z_0, \varepsilon )
$$
for some $g \in {\mathcal F}_\Omega (c )$, and $\varepsilon \in \partial {\mathbb D}$
if, and only if, $g(z) \equiv P( \omega_{\gamma , \varepsilon } (z))$.
\end{proposition}

\begin{proof}
Let $g \in {\mathcal F}_\Omega (c )$ and $\omega = P^{-1} \circ g $.
Then by Lemma  \ref{lemma:Schur_polynomial_and_Caratheodoy_problem},
there exists $\omega^* \in H_1^\infty ({\mathbb D})$ such that
$$
  (P^{-1} \circ g)(z)  =  \omega (z)  =
  \frac{z \widetilde{A}_n(z) \omega^* (z) + \widetilde{B}_n(z)}{zA_n(z)\omega^* (z)+B_n(z)} ,
$$
where $A_n(z)$, $B_n(z)$, $\widetilde{A}_n(z)$ and $\widetilde{B}_n(z)$ are the Schur polynomials associated with $\gamma$.
Since $\omega^* \in H_1^\infty ({\mathbb D})$, we have
\begin{equation}\label{eq:p001}
|z \omega^* (z)| =
    \left|  \frac{B_n(z)\omega (z) - \widetilde{B}_n(z) }{\widetilde{A}_n(z) - A_n(z) \omega (z)}  \right|
    \leq  |z| ,\quad z \in {\mathbb D}.
\end{equation}

\begin{figure}[b]
\begin{tikzpicture}[scale=0.8]
\draw[very thick](3,3) circle (3); %
\node at (4.7,6) {{\Large ${\mathbb D}$}};
\draw[very thick] (4,4) circle (1);
\node at (3.9,2.7) {{${\mathbb D}(\rho(z ),r(z ))$}};
\draw[fill] (3.2,4.6) circle (1pt);
\node at (2.8,4.7) {$z ^*$};
\draw[very thick] plot [smooth]
coordinates { (13,8.3)(11,6.3)(9,2)(12,0.3)(14.5,0.2) };
\node at (13,7) {{\Large $\Omega$}};
\draw[->] plot [smooth]
coordinates {(6,4.5)(8,5)(9.5,5)};
\node at (8,5.5){$P$};
\draw[->] plot [smooth]
coordinates {(6.2,2.5)(8,2.3)(9.5,3)};
\node at (8,1.8) {$g$};
\draw[->] plot [smooth cycle]
coordinates {(13.5,5)(11,4)(10.5,2)(13.2,3.5)};
\node at (13,5.7) {$P({\mathbb D}(\rho(z ),r(z )))$};
\draw[fill] (11,3) circle (1pt);
\node at (11,2.5) {$g(z )$};
\draw[fill] (13.2,3.5) circle (1pt);
\node at (13.5,3.4) {$w^*$} ;
\draw[->] plot [smooth]
coordinates {(13.2,3.5)(14,4.5)};
\node at (14.2,4.5) {$v$};
\draw[dashed,->] (11,3)--(13.2,3.5);
\node at (3,-1) {$\left\{
  \begin{array}{l}
   z ^* = \rho(z )+r(z )e^{i\theta} \\
   w^* = P(z ^*) \\
  v = i r(z )e^{i\theta}P'(z ^*) \\
  \end{array}
\right.$} ;
\end{tikzpicture}
\caption{}\label{fig-1}
\end{figure}

It follows from Lemmas \ref{lemma:determinat_of_Mobius_transformation} and \ref{lemma:denominator_has_no_zeros}
that the inequality (\ref{eq:p001}) is equivalent to
$$
| \omega (z) - \rho (z) | \leq r (z) ,
$$
where
\begin{align*}
\rho (z) &=
   \frac{\overline{B_n(z)} \widetilde{B}_n(z) -|z|^2 \overline{A_n(z)} \widetilde{A}_n(z)}{ |B_n(z)|^2 -|z|^2 | A_n(z)|^2 } ,\\[2mm]
r(z) &=
    \frac{|z||B_n(z) \widetilde{A}_n(z) - A_n(z) \widetilde{B}_n(z)|}{ |B_n(z)|^2 -|z|^2 | A_n(z)|^2 }\\[2mm]
    &=
     \frac{|z|^{n+1} \prod_{k=0}^n (1-|\gamma_k|^2) }{ |B_n(z)|^2 -|z|^2 | A_n(z)|^2 }.
\end{align*}

Thus  $(P^{-1} \circ g)(z) = \omega (z) \in \overline{\mathbb D} ( \rho(z), r(z))$,
and so $g(z) \in P(\overline{\mathbb D} ( \rho(z), r(z)) )$ for any $g \in {\mathcal F}_\Omega (c )$.
Since a convex univalent function maps any closed subdisk of ${\mathbb D}$ onto a convex closed Jordan domain
with an  analytic convex boundary curve, for any $z \in {\mathbb D} \backslash \{ 0 \} $ and
$\theta \in {\mathbb R}$, $g(z)$ belongs to the left half plane
of the tangential line at $P(\rho(z)+r(z)e^{i \theta })$ with the tangential vector
$ir(z)e^{i \theta} P'(\rho(z)+r(z)e^{i \theta })$ (see Figure \ref{fig-1}). Thus
\begin{equation}\label{ineq:supporting_half_plane}
{\rm Re \,}
  \left\{  \frac{P(\rho(z)+r(z)e^{i \theta}) - g(z)}{r(z)e^{i \theta} P'(\rho(z)+r(z)e^{i \theta})} \right\} \geq  0  .
\end{equation}
Since the tangential line intersects the boundary curve only at $P(\rho(z)+r(z)e^{i \theta})$,
equality in (\ref{ineq:supporting_half_plane}) holds if, and only if, $g(z) = P(\rho (z)+ r(z)e^{i \theta})$.

\bigskip

If $g(z) = (P \circ \omega_{\gamma, \varepsilon })(z)$ with $| \varepsilon | =1$,
 Lemma \ref{lemma:Schur_polynomial_and_Caratheodoy_problem} shows that
$$
\omega (z) = \omega_{\gamma, \varepsilon}(z)
    = \frac{\varepsilon z \widetilde{A}_n(z)  + \widetilde{B}_n(z)}{\varepsilon zA_n(z)  + B_n (z) } .
$$
Hence by Lemma \ref{lemma:determinat_of_Mobius_transformation}
\begin{align}
&
\omega_{\gamma, \varepsilon}(z) - \rho(z)
\label{eq:tilde_omega-center}
\\
=&
\frac{\varepsilon z \widetilde{A}_n(z)  + \widetilde{B}_n(z)}{\varepsilon zA_n(z) + B_n (z) }
    - \frac{\overline{B_n(z)} \widetilde{B}_n(z)-|z|^2 \overline{A_n(z)} \widetilde{A}_n(z)}{ |B_n(z)|^2 -|z|^2 | A_n(z)|^2 }
\nonumber
\\[2mm]
=&
   \frac{\varepsilon z (B_n(z) \widetilde{A}_n(z) - A_n(z) \widetilde{B}_n(z) )}{ |B_n(z)|^2 -|z|^2 | A_n(z)|^2 }
   \cdot \frac{\; \overline{B_n(z)+ \varepsilon zA_n(z) } \; }{B_n(z)+ \varepsilon z A_n(z)}
\nonumber
\\[2mm]
=&
   \frac{\prod_{k=0}^n (1-|\gamma_k|^2) |B_n(z)+ \varepsilon zA_n(z)|^2 }{ |B_n(z)|^2 -|z|^2 | A_n(z)|^2 }
   \cdot \frac{\varepsilon z^{n+1}}{(B_n(z)+\varepsilon zA_n(z) )^2} .
\nonumber
\end{align}
In particular, $| \omega_{\gamma , \varepsilon }(z) - \rho(z)| = r(z)$.
Thus for any $z \in {\mathbb D} \backslash \{ 0\}$, there exists $\theta \in {\mathbb R}$ such that
$\omega_{\gamma , \varepsilon }(z) = \rho(z) + r(z)e^{i \theta }$.
Substituting this into (\ref{ineq:supporting_half_plane}) we have
\begin{equation}\label{ineq:fundamental}
   {\rm Re \,}
  \left\{
  \frac{P(\omega_{\gamma , \varepsilon }(z))  - g(z)}
   {(\omega_{\gamma , \varepsilon }(z) - \rho (z))
   P'(\omega_{\gamma , \varepsilon }(z))}
   \right\}
   \geq  0,
\end{equation}
with equality if, and only if, $g(z) = P(\omega_{\gamma , \varepsilon }(z))$.
Rewriting  (\ref{ineq:fundamental}) and using (\ref{eq:tilde_omega-center})
we obtain
$$
{\rm Re} \,
 \left\{
  \frac{z^j( P(\omega_{\gamma , \varepsilon }(z)) -P(c_0)) - z^j(g(z)-g(0))}
  {
  \prod_{k=0}^n (1-|\gamma_k|^2)
  \displaystyle
  \frac{|B_n(z)+ \varepsilon zA_n(z)|^2}{|B_n(z)|^2 -|z|^2 | A_n(z)|^2}
  \displaystyle
  \frac{\varepsilon z^{n+j+1} P'(\omega_{\gamma , \varepsilon }(z) ) }{(B_n(z)+ \varepsilon zA_n(z) )^2}
  }
  \right\}
  \geq  0
$$
for $z \in {\mathbb D} \backslash \{ 0 \}$, and so by Lemma \ref{lemma:denominator_has_no_zeros} we have
\begin{equation}\label{ineq:key_ineq}
{\rm Re \,}
 \left\{  \frac{z^j( P(\omega_{\gamma , \varepsilon }(z)) -P(c_0)) - z^j(g(z)-g(0))}{\varepsilon h'(z)}  \right\}
  \geq 0
\end{equation}
for $z \in {\mathbb D}\backslash \{ 0 \}$, where
\begin{equation*}\label{eq:definition_of_k}
h(z) =
    \int_0^z \frac{\zeta^{n+j+1} P'(\omega_{\gamma, \varepsilon }(\zeta )) }{(B_n( \zeta )+ \varepsilon \zeta A_n( \zeta ) )^2} \, d \zeta .
\end{equation*}

We now claim that
\begin{equation}\label{ineq:convexity_of_k}
 {\rm Re \,} \Big(1+ \frac{z h''(z)}{h'(z)} \Big) > 0,  \quad\mbox{for}\quad z \in {\mathbb D}.
\end{equation}
Assuming the inequality (\ref{ineq:convexity_of_k}) for the moment, we complete the proof.
\bigskip

Notice that $h$ has  a Taylor series representation of the form $h(z)=az^{n+j+2}+\cdots$,
with $a=(n+j+2)^{-1} P'(c_0)$. By Lemma \ref{lemma:p-valent_convexity},
there exists $h_0 \in \mathcal{S}^*$ such that $h(z) = a h_0(z)^{n+j+2}$.
Since $h_0$ is starlike, for any $z_0\in\mathbb{D}\backslash \{0\}$, the line segment joining
$0$ and $h_0(z_0)$ lies entirely in $h_0(\mathbb{D})$.
Define a curve $\Gamma : z=z(t)$, $0 \leq t \leq 1$ joining $0$ to $z_0$ by
$z(t) = h_0^{-1}(t^{1/(n+j+2)} h_0(z_0))$. Then
$$
   h(z(t))
   =
   \frac{P'(c_0)}{n+j+2}
   h_0( h_0^{-1} (t^{1/(n+j+2)} h_0(z_0) ) )^{n+j+2}
   =
   t h(z_0) .
$$
Thus $h'(z(t)) z'(t) \equiv h(z_0)$ for $0 < t \leq 1$. This, and (\ref{ineq:key_ineq}), shows that
\begin{align}
  0 \leq&
   \int_0^1 {\rm Re \,} \
   \left[
   \frac{\{ z(t)^j( P(\omega_{\gamma, \varepsilon }(z(t)))-P(c_0)) - z(t)^j (g(z(t))-g(0) ) \} z'(t) }{\varepsilon h'(z(t))z'(t)}
   \right] \, dt
\label{ineq:integration_on _Gamma}
\\[2mm]
  = \ &
   {\rm Re \,}
   \left[
  \frac{ \int_0^1   z(t)^j   \{ P(\omega_{\gamma, \varepsilon}(z(t)))-P(c_0) \} z'(t) \, dt
   - \int_0^1   z(t)^j   \{ g(z(t)) - g(0) \}z'(t) \, dt} {\varepsilon h(z_0)}
   \right]
\nonumber
\\[2mm]
  =\ &
   {\rm Re \,}
   \left[
  \frac{ Q_{\gamma , j}(z_0, \varepsilon ) - \int_0^{z_0} z^j \{ g(z) - g(0) \} \, dz} {\varepsilon h(z_0)}
   \right] .
\nonumber
\end{align}
This implies that for any $g \in {\mathcal F}_\Omega (c )$, the value of the integral
$\int_0^{z_0} z^j \{ g(z) - g(0) \} \, dz$ belongs to a closed half plane, i.e.,
$$
\int_0^{z_0}  z^j \{ g(z) - g(0) \} \, dz \in {\mathbb H}(w_0, \alpha )
   := \{ w \in {\mathbb C} : {\rm Re \,} \{(w_0 - w ) /\alpha \} \geq 0 \},
$$
where $w_0 = Q_{\gamma,j}(z_0,\varepsilon)$ and $\alpha = \varepsilon h(z_0 ) $. Thus
\begin{equation}\label{eq:V_is_contained_in_H}
V_\Omega^j(z_0,c ) \subset {\mathbb H}(w_0, \alpha ) .
\end{equation}
Since $P\circ \omega_{\gamma,\varepsilon} \in {\mathcal F}_\Omega (c )$, it follows that
\begin{equation}\label{eq:w0_is_in_boundary_of_H}
w_0 = Q_{\gamma,j}(z_0,\varepsilon)
\in \partial {\mathbb H}(w_0, \alpha )
\cap V_\Omega^j(z_0,c ),
\end{equation}
and so  from (\ref{eq:V_is_contained_in_H}) and (\ref{eq:w0_is_in_boundary_of_H}) we obtain
$Q_{\gamma,j}(z_0,\varepsilon)  = w_0 \in \partial V_\Omega^j(z_0,c )$.

\bigskip
We next deal with uniqueness. Suppose that
$$
\int_0^{z_0}z^j \{ g(z) -g(0) \} \, dz
= Q_{\gamma,j}(z_0,\varepsilon) = \int_0^{z_0} \zeta^j \{ P(\omega_{\gamma,\varepsilon}(\zeta )) -P(c_0) \} \, d \zeta
$$
holds for some $g \in {\mathcal F}_\Omega (c )$ and $\varepsilon \in \partial {\mathbb D}$.
Then from (\ref{ineq:integration_on _Gamma}) it follows that
$$
   {\rm Re \,}
 \Big(
  \frac{z^j( P(\omega_{\gamma,\varepsilon}(z)) -P(c_0))  -  z^j(g(z)-g(0))}  {\varepsilon h'(z)}
  \Big)
  = 0
$$
holds on the curve $\Gamma$, and so
$$
     {\rm Re \,}
  \Big(
  \frac{P(\omega_{\gamma,\varepsilon}(z))  - g(z)} {(\omega_{\gamma,\varepsilon}(z) - \rho (z)) P'(\omega_{\gamma,\varepsilon}(z) )}
   \Big)
  = 0
$$
on $\Gamma$. By the equality condition of (\ref{ineq:fundamental}), we have $g(z) = P (\omega_{\gamma,\varepsilon}(z) )$ on $\Gamma$,
and from the identity theorem for analytic functions it follows that $g= P \circ \omega_{\gamma,\varepsilon}$.

\bigskip
We finally prove  (\ref{ineq:convexity_of_k}).
\bigskip

Since the harmonic function
${\rm Re} ( z h''(z)/h'(z)) +1 $ assumes the value $n+j+2 \geq 1 > 0 $ at the origin,
by using the minimum principle it suffices to show $\text{\rm Re} ( z h''(z)/h'(z)) +1 \geq 0$
in ${\mathbb D}$.
\bigskip

As in the proof of Proposition \ref{prop:interior_point},
we may assume without loss of generality that
$$
P'(z) = \frac{P'(0)}{ \prod_{i=1}^m (1- \eta_i z )^{\beta_i }},
$$
where $\eta_i \in \partial {\mathbb D}$, $0< \beta_i \leq 2$ $(i=1,2, \ldots , m)$, and $\sum_{i=1}^m \beta_i = 2$.
Thus from Lemma \ref{lemma:Schur_polynomial_and_Caratheodoy_problem} we have
$$
   h(z) = \int_0^z \frac{P'(0) \zeta^{n+j+1}}
   {\prod_{i=1}^m \{ B_n(\zeta ) + \varepsilon \zeta A_n(\zeta) - \eta_i( \widetilde{B}_n( \zeta )
    + \varepsilon \zeta \widetilde{A}_n (\zeta)) \}^{\beta_i}}\, d \zeta
$$
and hence a simple computation gives
$$
    z \frac{h''(z)}{h'(z)}
     =
     n+j+1
     -
     \sum_{i=1}^m
     \beta_i z
     \frac{d}{dz}
     \log \{
     B_n(z ) + \varepsilon z A_n( z )
   - \eta_i
   ( \widetilde{B}_n( z )
     + \varepsilon z \widetilde{A}_n (z )) \} .
$$
We note that for each $|\varepsilon|=1$, the function $\omega_{\gamma , \varepsilon} (z) = \{\widetilde{B}_n( z ) + \varepsilon z \widetilde{A}_n ( z ) \} / \{B_n(z ) + \varepsilon z A_n( z ) \}$ is a finite Blaschke product. Thus for each $i$, the polynomials $q_i(z):=B_n(z ) + \varepsilon z A_n( z ) - \eta_i \{ \widetilde{B}_n( z ) + \varepsilon z \widetilde{A}_n ( z )\}$ have no zeros in ${\mathbb D}$, and are of degree  $n+1$ at most.
Since $B_n(0)=1$ and $\widetilde{B}_n(0) = \gamma_0$,  the polynomials $q_i(z)$ can be expressed as
$$
q_i(z)= (1- \eta_i\gamma_0) \prod_{\ell=1}^{p_i} \left( 1 - \frac{z}{z_{i \ell }} \right) ,
$$
where $z_{i1}, \ldots , z_{i p_i} \in {\mathbb C} \backslash {\mathbb D}$
are the zeros of the polynomial $q_i(z)$, and $0 \leq p_i \leq n+1$. Therefore  using the identity
$w/(1-w)= 2^{-1}\{(1+w)/(1-w)-1 \}$ we have
\begin{align*}
   {\rm Re \,} \left\{ z \frac{h''(z)}{h'(z)} \right\}
    =&
    n+j+1 + \sum_{i=1}^m \beta_i \sum_{\ell=1}^{p_i}
    {\rm Re \,}
    \left\{\frac{z/z_{i \ell}}{1- (z/z_{i \ell})} \right\}
\\
    =&
    n+j+1
    - \sum_{i=1}^m \frac{\beta_i p_i}{2}
    + \sum_{i=1}^m
    \frac{\beta_i}{2} \sum_{\ell=1}^{p_i}
    {\rm Re \,}
    \left\{ \frac{1+ (z/z_{i \ell})}{1- (z/z_{i \ell})} \right\}
\\
    \geq&
    j +
    \sum_{i=1}^m \frac{\beta_i}{2} \sum_{\ell=1}^{p_i}
    {\rm Re \,}
    \left\{ \frac{1+ (z/z_{i \ell})}{1- (z/z_{i \ell})} \right\}
\\
  \geq& j \geq -1 .
\end{align*}
\end{proof}

Now we are in a position to prove our main result.

\begin{proof}[\textbf{Proof of Theorem \ref{thm:Main_theorem}}]
Let $c  \in \mbox{\rm Int} \, {\mathcal C}(n)$, $z_0 \in {\mathbb D} \backslash \{ 0 \}$
and $\gamma = (\gamma_0, \ldots , \gamma_n)$ be the Schur parameter of $c $.
Corollary \ref{cor:compact_and_convex} and Proposition \ref{prop:interior_point},
shows that the set $V_\Omega^j(z_0, c )$ is a compact and convex subset of ${\mathbb C}$,
and $Q_{\gamma,j}(z_0,0)$ is an interior point of $V_\Omega^j(z_0, c )$.
From these properties it is not difficult to see that $\partial V_\Omega^j(z_0, c )$
is a Jordan curve, and $V_\Omega^j(z_0, c )$ is a union of $\partial V_\Omega^j ( z_0 , c )$
and its inside domain, i.e., $V_\Omega^j(z_0, c )$ is a closed Jordan domain.

\bigskip
Proposition \ref{prop:boundary_uniqueness} shows that the map
$\partial {\mathbb D} \ni \varepsilon \mapsto Q_{\gamma,j}(z_0,\varepsilon ) \in \partial V_\Omega^j(z_0, c)$
is a closed curve. Furthermore, it is a simple curve. Indeed, if
$Q_{\gamma,j}(z_0,\varepsilon_1 )  = Q_{\gamma,j}(z_0,\varepsilon_2 )$
for some $\varepsilon_1, \varepsilon_2 \in \partial {\mathbb D}$,
then by the uniqueness part of Proposition \ref{prop:boundary_uniqueness} we have
$P ( \omega_{\gamma,\varepsilon_1} (z)) \equiv P ( \omega_{\gamma, \varepsilon_2} (z))$,
and so $\omega_{\gamma,\varepsilon_1} (z) \equiv \omega_{\gamma,\varepsilon_2} (z)$.
Using the representation (\ref{eq:extremal_omega_and_schur_polynomials}) for
$\omega_{\gamma,\varepsilon}$, we have
$$
\frac{ \varepsilon_1 z\widetilde{A}_n(z)+ \widetilde{B}_n(z)}{ \varepsilon_1 z A_n(z) + B_n(z)}
=\frac{ \varepsilon_2 z\widetilde{A}_n(z)+ \widetilde{B}_n(z)}{ \varepsilon_2 z A_n(z) + B_n(z)}
$$
which gives
$$
\varepsilon_1 z (\widetilde{A}_n(z)B_n(z)-A_n(z)\widetilde{B}_n(z))
=\varepsilon_2 z (\widetilde{A}_n(z)B_n(z)-A_n(z)\widetilde{B}_n(z)).
$$
Consequently, by Lemma \ref{lemma:determinat_of_Mobius_transformation} we conclude that $\varepsilon_1 = \varepsilon_2$.

\bigskip
Since a simple closed curve cannot contain any simple closed curve other than itself, the map
$\partial {\mathbb D} \ni \varepsilon \mapsto Q_{\gamma,j}(z_0,\varepsilon ) \in \partial V_\Omega^j(z_0, c)$
is surjective, and a parametrization of the boundary curve $\partial V_\Omega^j(z_0, c )$.
It therefore follows from Darboux's theorem (see \cite[Lemma 1.1]{Pommerenke-book})
that for fixed $z_0 \in {\mathbb D} \backslash \{ 0 \}$, $Q_{\gamma,j}(z_0,\varepsilon )$ is a convex univalent analytic
function of $\varepsilon \in \overline{\mathbb D}$, and
$V_\Omega^j(z_0, c ) = \{Q_{\gamma,j}(z_0,\varepsilon ) : \varepsilon \in \overline{\mathbb D} \}$.
This completes the proof.
\end{proof}


\end{document}